\documentclass{amsart}

\usepackage{amsmath,amsfonts,amssymb,amscd}

\def\Z{{\mathbb Z}}

\def\F{{\mathbb F}}

\def\CC{{\mathcal C}}

\def\cl{\mathrm{cl}}

\def\supp{\mathrm{supp}}
\def\div{\mathrm{div}}

\def\fchar{\mathrm{char}}

                             \def\cl{\mathrm{cl}}
                             \def\supp{\mathrm{supp}}

     \def\W{{\mathfrak W}}

                          \def\rr{{\mathfrak r}}
                          \def\RR{{\mathfrak R}}
\def\a{{\mathfrak a}}
                               \def\b{{\mathfrak b}}

\newtheorem{thm}{Theorem}[section]
\newtheorem{lem}[thm]{Lemma}
\newtheorem{cor}[thm]{Corollary}

\theoremstyle{definition}

\newtheorem{ex}[thm]{Example}

\newtheorem{sect}[thm]{}

           \newtheorem{rem}[thm]{Remark}

\begin{document}\date{}

\title[Halves of points]{Halves of points of an  odd degree hyperelliptic curve in its jacobian}
\author {Yuri G. Zarhin}

\address{Pennsylvania State University, Department of Mathematics, University Park, PA 16802, USA}

\email{zarhin@math.psu.edu}
\thanks{Partially supported by  Simons Foundation  Collaboration grant  \# 585711. \newline
This paper was started
 during my stay in May-July 2018  at the Max-Planck-Institut f\"ur Mathematik (Bonn, Germany), whose hospitality and support are gratefully acknowledged.}



\begin{abstract}
Let $f(x)$ be a degree $(2g+1)$ monic polynomial with coefficients in an  algebraically closed field $K$ with $\fchar(K)\ne 2$ and without repeated roots. Let $\RR\subset K$ be the $(2g+1)$-element set of  roots of  $f(x)$.
Let $\CC: y^2=f(x)$ be an odd degree genus $g$ hyperelliptic curve over $K$.
Let $J$ be the jacobian of $\CC$ and $J[2]\subset J(K)$ the (sub)group of points of order dividing $2$.
We identify $\CC$ with the image of its canonical embedding into $J$ (the infinite point of $\CC$ goes to the identity element of $J$).  Let $P=(a,b)\in \CC(K)\subset J(K)$ and
$$M_{1/2,P}=\{\a \in J(K)\mid 2\a=P\}\subset J(K),$$ which is $J[2]$-torsor. In a previous work we established an explicit bijection between the sets $M_{1/2,P}$ and
$$\RR_{1/2,P}:=\{\rr: \RR\to K\mid \rr(\alpha)^2=a-\alpha \ \forall \alpha\in\RR; \ 
 \prod_{\alpha\in\RR}\rr(\alpha)=-b\}.$$
The aim of this paper is to describe the induced action of $J[2]$ on $\RR_{1/2,P}$ (i.e., how  signs of  square roots $r(\alpha)=\sqrt{a-\alpha}$ should change).
\end{abstract}

\maketitle

\section{Introduction}

Let $K$ be an algebraically closed field of characteristic different from 2, $g$ a positive integer, $\RR\subset K$ a $(2g+1)$-element set,
$$f(x)=f_{\RR}(x):=\prod_{\alpha\in \RR}(x-\alpha)$$ a degree $(2g+1)$ polynomial with coefficients in $K$ and without repeated roots, 
$\CC:y^2=f(x)$ the corresponding genus $g$ hyperelliptic curve over $K$, and $J$ the jacobian of $\CC$. We identify $\CC$ with the image of its canonical embedding 
$$\CC\hookrightarrow J, \ P \mapsto \cl((P)-(\infty))$$
into $J$ (the infinite point $\infty$ of $\CC$ goes to the identity element of $J$). 
Let $J[2] \subset J(K)$ be the kernel of multiplication by $2$ in $J(K)$, which is a $2g$-dimensional $\mathbb{F}_2$-vector space.  All the $(2g+1)$ points 
$$\W_{\alpha}:=(\alpha,0)\in \CC(K)\subset J(K) \  (\alpha\in\RR)$$
lie in $J[2]$ and generate it as the $2g$-dimensional $\mathbb{F}_2$-vector space; they satisfy the only relation
$$\sum_{\alpha\in\RR}\W_{\alpha}=0 \in J[2]\subset J(K).$$
This leads to a well known canonical isomorphism \cite{Mumford} between $\mathbb{F}_2$-vector spaces $J[2]$ and
$$({\mathbb{F}_2}^{\RR})^0=\{\phi: \RR \to \F_2\mid \sum_{\alpha\in \RR}\phi(\alpha)=0\}.$$
 Namely, each function  $\phi \in ({\mathbb{F}_2}^{\RR})^0$ corresponds to $$\sum_{\alpha\in\RR}\phi(\alpha)\W_{\alpha}\in J[2].$$
For example,  for each $\beta\in \RR$ the point
$\W_{\beta}=\sum_{\alpha\ne \beta}\W_{\alpha}$
corresponds to the function $\psi_{\beta}:\RR \to\F_2$ that sends $\beta$ to $0$ and all other elements of $\RR$ to $1$.

 If $\b \in J(K)$ then the finite set
$$M_{1/2,\b}:=\{\a \in J(K)\mid 2\a=\b\}\subset J(K)$$
consists of $2^{2g}$ elements and
carries the natural structure of a $J[2]$-torsor. 

Let
$$P=(a,b) \in \CC(K)\subset J(K).$$
Let us consider, the set 
$$\RR_{1/2,P}:=\{\rr: \RR\to K\mid \rr(\alpha)^2=a-\alpha \ \forall \alpha\in\RR; \ 
 \prod_{\alpha\in\RR}\rr(\alpha)=-b\}.$$ 
Changes of signs in the (even number of) square roots provide $\RR_{1/2,P}$ with the natural structure of a $({\mathbb{F}_2}^{\RR})^0$-torsor.  Namely, let
$$\chi: \F_2 \to K^{*}$$
be the {\bf additive character} such that
$$\chi(0)=1, \chi(1)=-1.$$
Then the result of the action of a function $\phi:\RR \to \F_2$ from $({\mathbb{F}_2}^{\RR})^0$ on $\rr: \RR\to K$ from $\RR_{1/2,P}$ is just the product
$$\chi(\phi)\rr: \RR \to K, \ \alpha \mapsto \chi(\phi(\alpha))\rr(\alpha).$$
On the other hand, I constructed in \cite{ZarhinHyper2} an explicit {\sl bijection} of finite sets  
$$\RR_{1/2,P} \cong M_{1/2,P}, \ \rr\mapsto \a_{\rr}\in M_{1/2,P}\subset J(K).$$
Identifying (as above) $J[2]$ and $({\mathbb{F}_2}^{\RR})^0$, we obtain a second
structure of a  $({\mathbb{F}_2}^{\RR})^0$-torsor on $\RR_{1/2,P}$. Our main result asserts that these two structures actually coincide. In down-to-earth terms this means the following.

\begin{thm}
\label{signMain}
Let $\rr \in \RR_{1/2,P}$ and $\beta \in \RR$. Let us define $\rr^{\beta}\in \RR_{1/2,P}$ as follows.
$$\rr^{\beta}(\beta)=\rr(\beta), \ \rr^{\beta}(\alpha)=-\rr(\alpha) \ \forall \alpha \in \RR\setminus \{\beta\}.$$
Then
$$\a_{\rr^{\beta}}=\a_{\rr}+\W_{\beta}=\a_{\rr}+\left(\sum_{\alpha\ne \beta}\W_{\alpha}\right).$$
\end{thm}

\begin{rem}
In the case of elliptic curves (i.e., when $g=1$) the assertion of Theorem \ref{signMain} was proven in \cite[Th. 2.3(iv)]{BZ}.
\end{rem}

\begin{ex}
\label{specialB}
If $P=\W_{\beta}=(\beta,0)$ then 
$$\a_{\rr}+\W_{\beta}=\a_{\rr}-\W_{\beta}=\a_{\rr}-2\a_{\rr}=-\a_{\rr}$$
while 
$$-\a_{\rr}=\a_{-\rr}$$
 (see \cite[Remark 3.5]{ZarhinHyper2}).
On the other hand,
$\rr(\beta)=\sqrt{\beta-\beta}=0$ for all $\rr$ and
$$\rr^{\beta}=-\rr: \alpha \mapsto -\rr(\alpha) \ \forall \alpha \in \RR.$$
This implies that 
$$\a_{\rr^{\beta}}=\a_{-\rr}=\a_{\rr}+\W_{\beta}.$$
This proves Theorem \ref{signMain} in the special case $P=\W_{\beta}$.
\end{ex}

The paper is organized as follows. In Section  \ref{divisors} we recall basic facts about 
Mumford representations of points of $J(K)$ and review results of \cite{ZarhinHyper2}, including an explicit description of the bijection between 
$\RR_{1/2,P}$ and $M_{1/2,P}$. 
In Section \ref{transW} we give explicit formulas for the Mumford representation of $\a+\W_{\beta}$ when
$\a$ lies neither on the theta divisor of $J$ nor on its translation by $\W_{\beta}$, assuming that we know the Mumford representation of $\a$. In Section \ref{quadrics} we prove Theorem \ref{signMain}, using auxiliary results from commutative algebra that are proven in Section \ref{useful}.

\section{Halves and square roots}
\label{divisors}

Let $\CC$ be the smooth projective model of the smooth affine plane $K$-curve
$$y^2=f(x)=\prod_{\alpha\in\RR}(x-\alpha)$$
where $\RR$ is a $(2g+1)$-element subset  of $K$. In particular, $f(x)$ is a monic degree $(2g+1)$ polynomial without repeated roots.
 It is well known that $\CC$ is a genus $g$ hyperelliptic curve over $K$ with precisely one {\sl infinite} point, which we denote by $\infty$.  In other words,
$$\CC(K)=\{(a,b)\in K^2\mid  b^2=\prod_{\alpha\in \RR}(a-\alpha_i)\}\bigsqcup \{\infty\} .$$
Clearly, $x$ and $y$ are nonconstant rational functions on $\CC$, whose only pole is $\infty$. More precisely, the polar divisor of $x$ is $2 (\infty)$ and the polar divisor of $y$ is $(2g+1)(\infty)$. The zero divisor of $y$ is $\sum_{\alpha\in\RR} (\W_{\alpha})$.
In particular, $y$ is a {\sl local parameter} at (every) $\W_{\alpha}$. 

 We write $\iota$ for the hyperelliptic involution 
$$\iota: \CC \to \CC,  (x,y)\mapsto (x,-y), \ \infty \mapsto\infty.$$
The set  of fixed points of $\iota$ consists of $\infty$ and all $\W_{\alpha}$ ($\alpha\in\RR$).
It is well known that for each $P \in \CC(K)$ the divisor $(P)+\iota(P)-2(\infty)$ is principal. 
More precisely, if $P=(a,b)\in \CC(K)$ then $(P)+\iota(P)-2(\infty)$ is the divisor of the rational function $x-a$ on $C$. 
In particular, if $P=\W_{\alpha}=(\alpha,0)$ then 
$$2(\W_{\alpha})-2(\infty)=\div(x-\alpha).$$
In particular, $x-\alpha$ has a double zero at $\W_{\alpha}$ (and no other zeros).
 If $D$ is a divisor on $\CC$ then we write $\supp(D)$ for its {\sl support}, which is a finite subset of $\CC(K)$.

Recall that  the jacobian $J$   of $\CC$ is a $g$-dimensional abelian variety over $K$.  If $D$ is a degree zero divisor on $\CC$ then we write $\cl(D)$ for its linear equivalence class, which is viewed as an element of $J(K)$. Elements of $J(K)$ may be described in terms of so called {\bf Mumford representations} (see  \cite[Sect. 3.12]{Mumford}, \cite[Sect. 13.2]{Wash} and Subsection \ref{MumfordRep} below).

We will identify $\CC$ with its image in $J$ with respect to the canonical regular map $\CC \hookrightarrow J$ under which  $\infty$ goes to 
the identity element of $J$. In other words, a point $P \in \CC(K)$ is identified with  $\cl((P)-(\infty))\in J(K)$. Then the action of 
the hyperelliptic involution $\iota$ on  $\CC(K)\subset J(K)$ coincides with multiplication by $-1$ on $J(K)$.
In particular, the list of points of order  2 on $\CC$ consists of  all $\W_{\alpha}$ ($\alpha\in\RR$).

\begin{sect}
\label{biject}
Since $K$ is algebraically closed, the commutative group $J(K)$ is divisible. It is well known that for each $\b \in J(K)$ there are exactly $2^{2g}$ elements $\a \in J(K)$ such that $2\a=\b$. In \cite{ZarhinHyper2} we established explicitly the following bijection $\rr \mapsto \a_{\rr}$ between the $2^{2g}$-element  sets $\RR_{1/2,P}$ and $M_{1/2,P}$.

If $\rr\in \RR_{1/2,P}$ then for each positive integer $i \le 2g+1$ let us consider $\mathbf{s}_i(\rr)\in K$ defined as
the value of  $i$th basic symmetric function  at $(2g+1)$ elements $\{\rr(\alpha)\mid \alpha\in \RR\}$ (notice that all $\rr(\alpha)$ are distinct, since their squares $\rr(\alpha)^2=a-\alpha$ are distinct). Let us consider the degree $g$ monic polynomial
$$
U_{\rr}(x)=(-1)^g \left[(a-x)^g+\sum_{j=1}^g \mathbf{s}_{2j}(\rr)(a-x)^{g-j}\right],$$
and the polynomial
$$V_{\rr}(x)=\sum_{j=1}^g \left(\mathbf{s}_{2j+1}(\rr)-\mathbf{s}_1(\rr)\mathbf{s}_{2j}(\rr)\right)(a-x)^{g-j}$$
whose degree is {\sl strictly less} than $g$.
Let $\{c_1, \dots, c_g\}\subset K$ be the collection of all $g$ roots of $U_{\rr}(x)$, i.e.,
$$U_{\rr}(x)=\prod_{j=1}^g(x-c_j) \in K[x].$$
Let us put 
$$d_j=V_{\rr}(c_j) \ \forall j=1, \dots, g.$$
It is proven in \cite[Th. 3.2]{ZarhinHyper2} that  $Q_j=(c_j,d_j)$ lies in $\CC(K)$ for all $j$ and
$$\a_{\rr}:=\cl\left(\left(\sum_{j=1}^g (Q_j)\right)-g(\infty)\right)\in J(K)$$
satisfies $2\a_{\rr}=P$, i.e., $\a_{\rr}\in M_{1/2,P}$. In addition, {\bf none of $Q_j$ coincides with any $\W_{\alpha}$}, i.e., 
$$U_{\rr}(\alpha)\ne 0, \ c_j \ne \alpha, \  d_j \ne 0.$$
The main result of \cite{ZarhinHyper2} asserts that the map
$$\RR_{1/2,P} \to M_{1/2,P}, \ \rr \mapsto \a_{\rr}$$ is a {\bf bijection}. 
\end{sect}

\begin{rem}
\label{UVtoRR}
Notice that one may express explicitly $\rr$ in terms of $U_{\rr}(x)$ and $V_{\rr}(x)$. Namely \cite[Th. 3.2]{ZarhinHyper2}, {\sl none} of $\alpha \in \RR$ is a root of $U_{\rr}(x)$ and
\begin{equation}
\label{UVtoRR1}
\rr(\alpha)=\mathbf{s}_1(\rr)+(-1)^g\frac{V_{\rr}(\alpha)}{U_{\rr}(\alpha)}
 \  \text{ for all }  \alpha \in \RR.
 \end{equation}
In order to determine $\mathbf{s}_1(\rr)$, let us fix two {\sl distinct}  roots $\beta,\gamma \in \RR$.  Then \cite[Cor. 3.4]{ZarhinHyper2}
$$\frac{V_{\rr}(\gamma)}{U_{\rr}(\gamma)}\ne \frac{V_{\rr}(\beta)}{U_{\rr}(\beta)}$$
and
\begin{equation}
\label{UVtoRR2}
\mathbf{s}_1(\rr)=\frac{(-1)^g}{2}\times \frac{\left(\beta +\left(\frac{V_{\rr}(\beta)}{U_{\rr}(\beta)}\right)^2\right)-\left(\gamma +\left(\frac{V_{\rr}(\gamma)}{U_{\rr}(\gamma)}\right)^2\right)}{ \frac{V_{\rr}(\gamma)}{U_{\rr}(\gamma)}-\frac{V_{\rr}(\beta)}{U_{\rr}(\beta)}}.
\end{equation}
\end{rem}

\begin{sect}
\label{MumfordRep} {\bf Mumford representations}  (see \cite[Sect. 3.12]{Mumford}, \cite[Sect. 13.2, pp. 411--415, especially, Prop. 13.4, Th. 13.5 and Th. 13.7]{Wash}). 
Recall \cite[Sect. 13.2, p. 411]{Wash} that if $D$ is an effective divisor on $\CC$ of (nonnegative) degree $m$, whose support does {\sl not} contain $\infty$, then the degree zero divisor $D-m(\infty)$ is called {\sl semi-reduced} if it enjoys the following properties.

\begin{itemize}
\item
If $\W_{\alpha}$ lies in $\supp(D)$ then it appears in $D$ with multiplicity 1.
\item
If a   point $Q$ of $\CC(K)$ lies in $\supp(D)$ and does not coincide with any of $\W_{\alpha}$  then $\iota(Q)$ does {\sl not} lie in $\supp(D)$.
\end{itemize}
If, in addition, $m \le g$ then $D-m(\infty)$ is called {\sl reduced}.

It is known (\cite[Ch. 3a]{Mumford}, \cite[Sect. 13.2, Prop. 3.6 on p. 413]{Wash}) that for each $\a \in J(K)$ there exist {\sl exactly one}  nonnegative  $m$ and  (effective) degree $m$ divisor  $D$  such that the degree zero divisor $D-m(\infty)$ is {\sl reduced} and  $\cl(D-m(\infty))=\a$.    If 
$$m\ge 1, \ D=\sum_{j=1}^m(Q_j)\ \mathrm{ where }  \ Q_j=(a_j,b_j) \in \CC(K) \ \text{ for all } \  j=1, \dots , m$$
(here $Q_j$ do {\sl not} have to be distinct)
then the corresponding
$$\a=\cl(D-m(\infty))=\sum_{j=1}^m Q_j \in J(K).$$
The {\sl Mumford representation}
of $\a \in J(K)$ 
is the pair $(U(x),V(x))$ of polynomials $U(x),V(x)\in K[x]$ such that 
$$U(x)=\prod_{j=1}^m(x-a_j)$$
is a degree $m$ monic polynomial while $V(x)$ has degree $<m=\deg(U)$, the polynomial  $V(x)^2-f(x)$ is divisible by $U(x)$, and
$$b_j=V(a_j), \ Q_j=(a_j,V(a_j))\in \CC(K) \  \text{ for all }  \ j=1, \dots m.$$
(Here $(a_j,b_j)$ are as above.)
  Such a pair always exists,  is unique, and (as we have just seen) uniquely determines not only $\a$ but also divisors $D$  and $D-m(\infty)$.

Conversely, if $U(x)$ is a  monic polynomial of degree $m\le g$ and $V(x)$ a polynomial such that $\deg(V)<\deg(U)$ and $V(x)^2-f(x)$ is divisible by $U(x)$ 
then there exists exactly one  $\a=\cl(D-m(\infty))$ where $D-m(\infty)$ is a reduced divisor and $(U(x),V(x))$ is the Mumford representation  of 
$\a=\cl(D-m(\infty))$.
\end{sect}

\begin{sect}
\label{halfRed}
In the notation of Subsect. \ref{biject}, let us consider the effective degree $g$ divisor
$$D_{\rr}:=\sum_{j=1}(Q_j)$$
on $\CC$. Then $\supp(D_{\rr})$ (obviously) does contain {\sl neither} $\infty$ {\sl nor} any of $\W_{\alpha}$'s. It is proven in \cite[Th. 3.2]{ZarhinHyper2} that the divisor $D_{\rr}-g(\infty)$ is {\sl reduced} and the pair $(U_{\rr}(x), V_{\rr}(x))$ is the Mumford representation of 
$$\a_{\rr}:=\cl(D_{\rr}-g(\infty)).$$
In particular, if $Q\in C(K)$ lies in $\supp(D)$ (i.e., is one of  $Q_j$'s) then $\iota(Q)$ does {\sl not}.
\end{sect}

\begin{lem}
\label{key}
Let $D$ be an effective divisor on $\CC$ of  degree $m>0$ such that $m \le 2g+1$ and $\supp(D)$ does not contain $\infty$. Assume that the divisor $D-m(\infty)$ is principal.

\begin{enumerate}
\item[(1)]
Suppose that $m$ is odd.
 Then:
 
 \begin{itemize}
 \item[(i)]
  $m=2g+1$ and there exists exactly one polynomial $v(x)\in K[x]$ such that  the divisor of $y-v(x)$ coincides with $D-(2g+1)(\infty)$. In addition, $\deg(v)\le g$.
    \item[(ii)]
    If $\W_{\alpha}$ lies in  $\supp(D)$ then it appears in $D$ with multiplicity 1.
    \item[(iii)]
    If $b$ is a nonzero element of $K$ and  $P=(a,b) \in \CC(K)$ lies in  $\supp(D)$ then $\iota(P)=(a,-b)$ does not lie in  $\supp(D)$.
    \item[(iv)] $D-(2g+1)(\infty)$ is semi-reduced (but not reduced).
  \end{itemize}
  \item[(2)]
  Suppose that $m=2d$ is even. Then:
\begin{itemize}
 \item[(i)]
 there exists exactly one monic degree $d$ polynomial $u(x)\in K[x]$ such that  the divisor of $u(x)$ coincides with $D-m(\infty)$; 
\item[(ii)]
every point $Q \in \CC(K)$ appears in $D-m(\infty)$ with the same multiplicity as $\iota(Q)$;
\item[(iii)]
 every $\W_{\alpha}$ appears in $D-m(\infty)$ with even multiplicity.
\end{itemize}
  \end{enumerate}
\end{lem}

\begin{proof}
All the assertions except (2)(iii) are already proven in \cite[Lemma 2.2]{ZarhinHyper2}.
In order to prove the remaining one, let us split the polynomial $v(x)$ into a product
$v(x)=(x-\alpha)^{d}v_1(x)$ where $d$ is a nonnegative integer and 
$v_1(x)\in K[x]$ satisfies $v_1(\alpha) \ne 0$.  Then $\W_{\alpha}$ appears in 
$D-m(\infty)$ with multiplicy $2d$, because $(x-\alpha)$ has a double zero at $\W_{\alpha}$. (See also \cite{StollHC}.)
\end{proof}

Let $d \le g$ be a positive integer and $\Theta_d \subset J$ be the image of the regular map
$$\CC^{d} \to J, \ (Q_1, \dots , Q_{d}) \mapsto \sum_{i=1}^{d} Q_i\subset J.$$
It is well known that $\Theta_d$ is an irreducible closed $d$-dimensional subvariety of $J$ that coincides with $\CC$ for $d=1$ and with $J$ if $d =g$; in addition, $\Theta_d\subset\Theta_{d+1}$ for all $d<g$. Clearly, each $\Theta_d$ is stable under multiplication by $-1$ in $J$.
We write $\Theta$ for the $(g-1)$-dimensional {\sl theta divisor} $\Theta_{g-1}$.  

\begin{thm}[See Th. 2.5 of \cite{ZarhinHyper2}]
\label{notheta}
Suppose that $g>1$ and 
let $$\CC_{1/2}:=2^{-1}\CC \subset J$$
 be the preimage of $\CC$ with respect to multiplication by 2 in $J$.  Then the intersection of $\CC_{1/2}(K)$ and $\Theta$ 
consists of points of order dividing  $2$ on $J$. In particular, the intersection of $\CC$ and $\CC_{1/2}$ consists of $\infty$ and all $\W_{\alpha}$'s. 
\end{thm}

\section{Adding Weierstrass points}
\label{transW}
In this section we discuss how to compute a sum $\a+\W_{\beta}$ in $J(K)$ when $\a\in J(K)$  lies neither on $\Theta$ nor on its translation $\Theta+\W_{\beta}$.
  Let $D-g(\infty)$ be the reduced divisor on $\CC$, whose class represents $\a$. Here 
$$D=\sum_{j=1}^g(Q_j)\ \mathrm{ where }  \ Q_j=(a_j,b_j) \in \CC(K)\setminus\{\infty\}$$
is a degree $g$ effective divisor. Let
 $(U(x),V(x))$ be the Mumford representation  of  $\cl(D-g(\infty))$. We have
 $$\deg(U)=g>\deg(V)$$, 
$$U(x)=\prod_{j=1}^g(x-a_j), \  b_j=V(a_j) \  \forall j$$
and $f(x)-V(x)^2$ is divisible by $U(x)$.
\begin{ex}
\label{plusWi}
Assume additionally that {\sl none} of $Q_j$ coincides with   $\W_{\beta}=(\beta,0)$, i.e.,  
$$U(\beta)\ne 0.$$
  Let us find explicitly the Mumford representation $(U^{[\beta]}(x),V^{[\beta}](x))$ of the sum
$$\a+\W_{\beta}=\cl(D-m(\infty))+\cl((\W_{\beta})-(\infty))=\cl((D+(\W_{\beta}))-(g+1)(\infty))=\cl(D_1-(g+1)(\infty)).$$
where 
$$D_1:=D+(\W_{\beta})=\left(\sum_{j=1}^g(Q_j)\right)+(\W_{\beta})$$
is a degree $(g+1)$ effective divisor on $\mathcal{C}$. (We will see that $\deg(\tilde{U}^{[\beta]})=g$.) Clearly, $D_1-(g+1)(\infty)$ is {\sl semi-reduced} but {\sl not} reduced.

Let us consider
the polynomials
$$U_1(x)=(x-\beta)U(x), \ V_1(x)=V(x)-\frac{V(\beta)}{U(\beta)}U(x)\in K[x].$$
Then $U_1$ is a degree $(g+1)$ monic polynomial, $\deg(V_1)\le g$,
$$V_1(\beta)=0,  \ V_1(a_j)=V(a_j)=b_j \ \forall j$$
and $f(x)-V_1(x)^2$ is divisible by $U_1(x)$. (The last assertion follows from  the divisibility of both $f(x)$ and $V_1(x)$ by $x-\beta$ combined with the divisibility of $f(x)-V(x)^2$  by $U(x)$.)
If we put 
$$a_{g+1}=\beta, \ b_{g+1}=0, \ Q_{g+1}=\W_{\beta}=(\beta,0)$$ then $$U_1(x)=\prod_{j=1}^{g+1}(x-a_j), \
D_1=\sum_{j=1}^{g+1}(Q_j)\ \mathrm{ where }  \ Q_j=(a_j,b_j) \in \CC(K), \ b_j=V_1(a_j) \forall j$$
and $f(x)-V_1(x)^2$ is divisible by $U_1(x)$. In particular, $(U_1(x),V_1(x))$ is the   pair of polynomials  that corresponds to semi-reduced $D_1-(g+1)(\infty)$  as described in \cite[Prop. 13.4 and Th. 3.5]{Wash}. In order to find the Mumford representation of 
$\cl(D_1-(g+1)(\infty))$,
we use an algorithm described in \cite[Th. 13.9]{Wash}. Namely, let us put
$$\tilde{U}(x)=\frac{f(x)-V_1(x)^2}{U_1(x)}\in K[x].$$
Since $\deg(V_1(x)) \le g$ and $\deg(f)=2g+1$, we have
$$\deg\left(V_1(x)^2\right)\le 2g, \  \deg\left(f(x)-V_1(x)^2\right)=2g+1, \deg\left(\tilde{U}(x)\right)=g.$$
Since $f(x)$ is monic, $f(x)-V_1(x)^2$ is also monic and therefore $\tilde{U}(x)$ is also monic, because $U_1(x)$ is monic.
By \cite[Th. 13.9]{Wash}, $U^{[\beta]}(x)=\tilde{U}(x)$ (since the latter is monic and has degree $g\le g$) and  $V^{[\beta]}(x)$
 is the {\sl remainder} of $-V_1(x)$ with respect to division by $\tilde{U}(x)$. Let us find this remainder. We have
$$-V_1(x)=-\left(V(x)-\frac{V(\beta)}{U(\beta)}U(x)\right)=-V(x)+\frac{V(\beta)}{U(\beta)}U(x).$$
Recall that
$$\deg(V)<g=\deg(U)=\deg(\tilde{U}).$$
This implies that the coefficient of $-V_1$ at $x^g$ equals $V(\beta)/U(\beta)$ and therefore
$$V^{[\beta]}(x)=\left(-V(x)+\frac{V(\beta)}{U(\beta)}U(x)\right)-\frac{V(\beta)}{U(\beta)}\tilde{U}(x)=-V(x)+\frac{V(\beta)}{U(\beta)}\left(U(x)-\tilde{U}(x)\right).$$
Using formulas above for $U_1,V_1,\tilde{U}$, we obtain that
\begin{equation}
\label{UWi}
U^{[\beta]}(x)=\frac{f(x)-\left(V(x)-\frac{V(\beta)}{U(\beta)}U(x)\right)^2}{(x-\beta)U(x)},
\end{equation}
\begin{equation}
\label{VWi}
V^{[\beta]}(x)=-V(x)+\frac{V(\beta)}{U(\beta)}\left(U(x)-
\frac{f(x)-\left(V(x)-\frac{V(\beta)}{U(\beta)}U(x)\right)^2}{(x-\beta)U(x)}     \right).
\end{equation}
\end{ex}

\begin{rem}
There is an algorithm of David Cantor \cite[Sect. 13.3]{Wash}   that explains how to compute the Mumford representation of a sum of arbitrary divisor classes (elements of $J(K)$) when their Mumford representations are given.
\end{rem}

\begin{rem}
Suppose that $\a \in J(K)$ and $P=2\a$ lies in $\CC(K)$ but is not the zero of the group law.  Then $\a$ does not lie on the theta divisor (Theorem \ref{notheta}) and satisfies the conditions of Example \ref{plusWi} for all $\beta\in\RR$ (see Subsect. \ref{biject}).
\end{rem}

\section{Proof of Main Theorem}
\label{quadrics}
Let us choose an order on $\RR$. This allows us to identify $\RR$ with $\{1, \dots, 2g, 2g+1\}$ and list elements of $\RR$ as $\{\alpha_1, \dots , \alpha_{2g}, \alpha_{2g+1}\}$.  Then
$$f(x)=\prod_{i=1}^{2g+1}(x-\alpha_i)$$ and the affine equation  for $\CC\setminus \{\infty\}$ is
$$y^2=\prod_{i=1}^{2g+1}(x-\alpha_i).$$
Slightly abusing notation, we denote $\W_{\alpha_i}$ by $\W_i$. 

Let us consider the closed affine $K$-subset $\tilde{\CC}$ in the affine $K$-space $\mathbb{A}^{2g+1}$  with coordinate functions ${z_1, \dots, z_{2g}, z_{2g+1}}$ that  is cut out by the  system of quadratic equations
$$z_1^2+\alpha_1=z_2^2+\alpha_2= \dots =z_{2g+1}^2+\alpha_{2g+1}.$$
We write $x$ for the regular function $z_i^2+\alpha_i$ on $\tilde{\CC}$, which does {\sl not} depend on a choice of $i$.
By Hilbert's Nullstellensatz, the $K$-algebra  $K[\tilde{\CC}]$ of regular functions on $\tilde{\CC}$ is canonically isomorphic to the following $K$-algebra. First, we  need to consider the quotient $A$ of the polynomial $K[x]$-algebra $K[x][T_1, \dots , T_{2g+1}]$ by the ideal generated by all quadratic polynomials $T_i^2 -(x-\alpha_i)$. Next, $K[\tilde{\CC}]$ is canonically isomorphic to the quotient $A/\mathcal{N}(A)$ where $\mathcal{N}(A)$ is the nilradical of $A$.
In the next section (Example \ref{irredTildeC}) we will prove that $A$ has no zero divisors (in particular, $\mathcal{N}(A)=\{0\}$) and  therefore $\tilde{\CC}$ is {\sl irreducible}. (See also \cite{Flynn}.)
We write $y$ for the regular function
$$y=-\prod_{i=1}^{2g}z_i \in K[\tilde{\CC}].$$
Clearly, $y^2=\prod_{i=1}^{2g}(x-\alpha_i)$ in $K[\tilde{\CC}]$. The pair $(x,y)$ gives rise to the finite regular map of affine $K$-varieties (actually, curves)
\begin{equation}
\label{mapH}
\mathfrak{h}: \tilde{\CC}\to \CC\setminus \{\infty\}, \ (r_1, \dots, r_{2g}, r_{2g+1})\mapsto  (a,b)=\left(r_1^2+\alpha_1,  - \prod_{i=1}^{2g+1}r_i\right)
\end{equation}
of degree $2^{2g}$. For each 
$$P=(a,b) \in K^2=\mathbb{A}^2(K) \ \mathrm{ with } \ b^2=\prod_{i=1}^{2g+1}(a-\alpha_i)$$
the fiber $\mathfrak{h}^{-1}(P)=\RR_{1/2,P}$ consists of (familiar) collections of square roots
$$\rr=\{r_i=\sqrt{a-\alpha_i}\mid 1\le i \le 2g+1\}$$ with $\prod_{i=1}^{2g+1} r_i=-b$.
Each such $\rr$ gives rise to $\a_{\rr}\in J(K)$ such that
$$2\a_{\rr}=P\in \CC(K)\subset J(K)$$
(see \cite[Th. 3.2]{ZarhinHyper2}). On the other hand, for each $\W_l=(\alpha_l,0)$  (with $1\le l\le 2g+1$) the sum
$\a_{\rr}+\W_{l}$ is also a half of $P$ and therefore corresponds to a certain collection of square roots. Which one? The answer is given by Theorem \ref{signMain}. We repeat its statement, using the new notation.

\begin{thm}
\label{torseur}
Let $P=(a,b)$ be a $K$-point on $\CC$ and 
$\rr=(r_1, \dots, r_{2g}, r_{2g+1})$ be a collection of square roots $r_i=\sqrt{a-\alpha_i}\in K$ such that $\prod_{i=1}^{2g+1}r_i=-b$.
Let $l$ be an integer that satisfies $1\le l\le 2g+1$ and let
\begin{equation}
\label{RRl1}
\rr^{[l]}=\left(r_1^{[l]}, \dots, r_{2g}^{[l]}, r_{2g+1}^{[l]}\right)\in \mathfrak{h}^{-1}(P)\subset \tilde{\CC}(K)
\end{equation}
 be the collection of square roots $r_i^{[l]}=\sqrt{a-\alpha_i}$ such  that
 \begin{equation}
\label{RRl2}
r_l^{[l]}=r_l, \  r_i^{[l]}=-r_i \ \forall \ i \ne l.
\end{equation}
Then
$$\a_{\rr}+\W_l=\a_{\rr^{[l]}}.$$
\end{thm}

\begin{ex}
\label{exW}
Let us take as $P$ the point $\W_l=(\alpha_l,0)$. Then 
$$r_l=\sqrt{\alpha_l-\alpha_l}=0 \ \forall \
\rr=(r_1, \dots, r_{2g}, r_{2g+1}) \in \mathfrak{h}^{-1}(\W_l)$$
and therefore
$$\rr^{[l]}=(-r_1, \dots, -r_{2g}, -r_{2g+1}) =-\rr.$$
It follows from Example \ref{specialB} (if we take $\beta=\alpha_l$) that 
$$ \a_{\rr}+\W_l=\a_{\rr}-\W_{l}=\a_{\rr}-2\a_{\rr}=-\a_{\rr}=\a_{\rr^{[l]}}.$$
This proves Theorem \ref{torseur}  in the case of $P=\W_l$. We are going to deduce the general case from this special one.
\end{ex}

\begin{sect}
\label{Teps}
Before starting the proof of Theorem \ref{torseur}, let us define for each collections of signs
$$\varepsilon =\{\epsilon_i=\pm 1\mid 1\le i \le 2g+1, \prod_{i=1}^{2g+1} \epsilon_i=1\}$$
the biregular automorphism
$$T_{\varepsilon}: \tilde{\CC} \to \tilde{\CC}, \ z_i \mapsto \epsilon_i z_i \ \forall i.$$
Clearly, all $T_{\varepsilon}$ constitute a finite automorphism  group of 
$\tilde{\CC}$ that leaves invariant  every $K$-fiber of $\mathfrak{h}: \tilde{\CC}\to \CC\setminus \{\infty\}$, acting on it {\bf transitively}.
Notice that if $T_{\varepsilon}$ leaves invariant all the points of a certain fiber $\mathfrak{h}^{-1}(P)$ with $P\in \CC(K)$ then all the $\epsilon_i=1$, i.e., $T_{\varepsilon}$ is the identity map.
\end{sect}

\begin{proof}[Proof of Theorem \ref{torseur}]
Let us put
$$\beta:=\alpha_l.$$
Then we have
$$\W_l=(\alpha_l,0)=(\beta,0).$$
Let us consider the automorphism (involution)
$$\mathfrak{s}^{[l]}: \tilde{\CC}\to  \tilde{\CC},  \ \rr\mapsto \rr^{[l]}$$   
of $\tilde{\CC}$ defined by \eqref{RRl1} and \eqref{RRl2}.
We need to define another (actually, it will turn out to be the same) involution (and therefore an automorphism)
$$\mathfrak{t}^{[l]}: \tilde{\CC}\to  \tilde{\CC}$$
that is defined  by 
$$\a_{\mathfrak{t}^{[l]}(\rr)} = \a_{\rr}+\W_l$$
 as  a composition of the following {\bf regular} maps. First,  $\rr\in \tilde{C}(K)$ goes to the pair of polynomials $(U_{\rr}(x),V_{\rr}(x))$  as in Remark \ref{UVtoRR}, which is the Mumford representation of $\a_{\rr}$ (see Subsect. \ref{halfRed}). Second, 
$(U_{\rr}(x),V_{\rr}(x))$ goes to the pair of polynomials $(U^{[\beta]}(x),V^{[\beta]}(x))$ defined by formulas \eqref{UWi} and \eqref{UWi} in Section \ref{transW}, which is the Mumford representation of $\a_{\rr}+\W_l$. 
Third, applying formulas \eqref{UVtoRR1}  and \eqref{UVtoRR2} in Remark \ref{UVtoRR}  to $(U^{[\beta]}(x),V^{[\beta}](x))$ (instead of $(U(x),V(x))$), we get at last $\mathfrak{t}^{[l]}(\rr) \in \tilde{\CC}(K)$ such that
$$\a_{\mathfrak{t}^{[l]}(\rr) }=\a_{\rr}+\W_l.$$
Clearly, $\mathfrak{t}^{[l]}$ is a regular selfmap of $\tilde{\CC}$  that is an involution, which implies that $\mathfrak{t}^{[l]}$ is a biregular automorphism of $\tilde{\CC}$. 
It is also clear that both $\mathfrak{s}^{[l]}$ and  $\mathfrak{t}^{[l]}$ leave invariant every fiber of $\mathfrak{h}: \tilde{\CC}\to \CC\setminus \{\infty\}$ 
and coincide on  $\mathfrak{h}^{-1}(\W_l)$, thanks to Example \ref{exW}. This implies that
$\mathfrak{u}:=\left(\mathfrak{s}^{[l]}\right)^{-1}\mathfrak{t}^{[l]}$ is a biregular automorphism of $\tilde{\CC}$ that leaves invariant every fiber of 
$\mathfrak{h}: \tilde{\CC}\to \CC\setminus \{\infty\}$ and acts as the identity map on $\mathfrak{h}^{-1}(\W_l)$.  
The invariance of each fiber of $\mathfrak{h}$ implies that $\tilde{\CC}(K)$ coincides with the finite union of its closed subsets $\tilde{\CC}_{\varepsilon}$ defined by the condition 
 $$\tilde{\CC}_{\varepsilon}:=\{Q \in \tilde{\CC}(K)\mid \mathfrak{u}(Q)=T_{\varepsilon}(Q)\}.$$
 Since $\tilde{\CC}$ is irreducible, the whole $\tilde{\CC}(K)$ coincides with one of $\tilde{\CC}_{\varepsilon}$. In particular, the fiber 
$$\mathfrak{h}^{-1}(\W_l)\subset \tilde{\CC}_{\varepsilon}$$
and therefore $T_{\varepsilon}$ acts identically on all points of $\mathfrak{h}^{-1}(\W_l)$. In light of arguments of Subsect. \ref{Teps}, $T_{\varepsilon}$ is the {\sl identity map} and therefore $\mathfrak{u}$ acts identically on the whole $\tilde{\CC}(K)$. This means that
 $\mathfrak{s}^{[l]}=\mathfrak{t}^{[l]}$, i.e.,
$$\a_{\rr}+\W_l=\a_{\rr^{[l]}}.$$
\end{proof}

\begin{sect}
\label{pointOrder2T}
Let $\phi: \RR \to \F_2$ be a function that satisfies $\sum_{\alpha\in\RR}\phi(\alpha)=0$, i.e. $\phi\in (\F_2^{\RR})^0$. Then the finite subset
$$\supp(\phi)=\{\alpha\in \RR\mid \phi(\alpha)\ne 0\}\subset\RR$$
has {\sl even} cardinality and the corresponding point of $J[2]$ is
$$\mathfrak{T}_{\phi}=\sum_{\alpha\in \RR}\phi(\alpha)\W_{\alpha}=
\sum_{\alpha\in \supp(\phi)}\W_{\alpha}=\sum_{\gamma\not\in \supp(\phi)}\W_{\gamma}.$$
\end{sect}

\begin{thm}
\label{add2}
Let $\rr\in \RR_{1/2,P}$. Let us define $\rr^{(\phi)} \in \RR_{1/2,P}$ as follows.
$$\rr^{(\phi)}(\alpha)=-\rr(\alpha) \ \forall \alpha \in \supp(\phi); \
\rr^{(\phi)}(\gamma)=\rr(\gamma) \ \forall \gamma \not\in \supp(\phi).$$
Then
$$\a_{\rr}+\mathfrak{T}_{\phi}=\a_{\rr^{(\phi)}}.$$
\end{thm}

\begin{rem}
If $\phi$ is identically zero then 
$$\mathfrak{T}_{\phi}=0\in J[2], \ \rr^{(\phi)}=\rr$$
and the assertion of Theorem \ref{add2} is obviously true.
 If $\alpha_l\in \RR$ and $\phi=\psi_{\alpha_l}$, i.e. $\supp(\phi)=\RR\setminus \{\alpha_l\}$ then
$$\mathfrak{T}_{\phi}=\W_l\in J[2], \ \rr^{(\phi)}=\rr^{[l]}$$
and the assertion of Theorem \ref{add2} follows from Theorem \ref{torseur}.
\end{rem}

\begin{proof}[Proof of Theorem \ref{add2}]
We may assume that $\phi$ is {\sl not} identically zero. We need to apply Theorem \ref{torseur} $d$ times where $d$ is the (even) cardinality of $\supp(\phi)$ in order to get $\rr^{\prime}\in\RR_{1/2,P}$ such that
$$\a_{\rr}+\sum_{\alpha\in \supp(\phi)}\W_{\alpha}=\a_{\rr^{\prime}}.$$
Let us check how many times do we need to change the sign of each $\rr(\beta)$.
First, if $\beta\not \in \supp(\phi)$ then we need to change to sign of $\rr(\beta)$ at every step, i.e., we do it exactly $d$ times. Since $d$ is even, the sign of $\rr(\beta)$ remains the same, i.e.,
$$\rr^{\prime}(\beta)=\rr(\beta) \ \forall \beta \not\in  \supp(\phi).$$
Now if $\beta \in \supp(\phi)$ then we need to change the sign of $\rr(\beta)$ every time when we add $W_{\alpha}$ with $\alpha \ne \beta$ and it occurs exactly $(d-1)$ times. On the other hand, when we add $\W_{\beta}$, we don't change the sign of $\rr(\beta)$. So, we change the sign of $\rr(\beta)$ exactly $(d-1)$ times, which implies that
$$\rr^{\prime}(\beta)=-\rr(\beta) \ \forall \beta \in  \supp(\phi).$$
Combining the last two displayed formula, we obtained that
$$\rr^{\prime}=\rr^{(\phi)}.$$
\end{proof}

\section{Useful Lemma}
\label{useful}
As usual, we define the Kronecker delta $\delta_{ik}$ as $1$ if $i=k$ and $0$ if $i\ne k$.

The following result is probably well known but I did not find a suitable reference. (However, see \cite[Lemma 5.10]{Flynn} and \cite[pp. 425--427]{BGG}.)

\begin{lem}
\label{fieldExt}
Let $n$ be a positive integer, $E$ a field provided with $n$ distinct discrete valuation maps
$$\nu_i: E^{*} \twoheadrightarrow \Z, \ (i=1, \dots, n).$$
For each $i$ let
 $O_{\nu_i} \subset E$  the discrete valuation ring attached to $\nu_i$
and  $\pi_i \in O_{\nu_i} $  its uniformizer, i.e., a generator of the maximal ideal in $O_{\nu_i}$.
Suppose that for each $i$ we are given a  prime number $p_i$  such that the characteristic of the residue field $O_{\nu_i} /\pi_i$ is different from $p_k$ for all $k \ne i$.
Let us assume also that
$$\nu_i(\pi_k)=\delta_{ik} \ \forall i,k=1, \dots n,$$
i.e, each $\pi_i$ is a $\nu_k$-adic unit if $i\ne k$. 

Then the  the quotient $B=E[T_1, \dots , T_n]/(T_1^{p_1}-\pi_1, \dots , T_n^{p_n}-\pi_n)$ of the polynomial $E$-algebra  $E[T_1, \dots T_n]$ by the ideal generated by all $T_i^{p_i}-\pi_i$ is a field that is an algebraic extension of $E$ of degree $\prod_{i=1}^n p_i$. 
In addition,  the set of monomials
$$S=\{\prod_{i=1}^n T_i^{e_i}\mid 0 \le e_i \le p_i-1\}\subset E[T_1, \dots T_n]$$
maps injectively into $B$ and its image is a basis of the $E$-vector space $B$.
\end{lem}

\begin{rem}
By definition of a uniformizer, $\nu_i(\pi_i)=1$ for all $i$.
\end{rem}

\begin{proof}[Proof of Lemma \ref {fieldExt}]
First, the cardinality of $S$ is $\prod_{i=1}^n p_i$ and the image of $S$ generates $B$ as the $E$-vector space. This implies that if the $E$-dimension of $B$ is $\prod_{i=1}^n p_i$  then the image of $S$ is a basis of the $E$-vector space $B$.
Second,  notice that for each $i$ the polynomial $T^{p_i}-\pi_i$ is  irreducible over $E$, thanks to the Eisenstein criterion applied to  $\nu_i$ and therefore $E[T_i]/(T^{p_i}-\pi_i)$ is a {\sl field} that is an algebraic degree $p_i$  extension of $E$. In particular, the $E$-dimension of $E[T_i]/(T^{p_i}-\pi_i)$ is $p_i$.
 This proves Lemma for $n=1$.

{\bf Induction by $n$}.  Suppose that $n>1$ and consider the finite  degree $p_i$ field extension $E_n=E[T_n]/(T^{p_n}-\pi_n)$ of $E$.

Clearly, the $E$-algebra $B$ is isomorphic to the quotient $E_n[T_1, \dots T_{n-1}]/(T_1^{p_1}-\pi_1, \dots , T_{n-1}^{p_{n-1}}-\pi_{n-1})$ of the polynomial ring $E_n[T_1, \dots T_{n-1}]$ by the ideal generated by all polynomials  $T_i^{p_i}-\pi_i$ with $i<n$.  Our goal is to apply the induction assumption to $E_n$ instead of $E$. In order to do that, let us consider for each $i <n$ the integral closure $\tilde{O}_i$ of $O_{\nu_i}$ in $E_n$. It is well known that $\tilde{O}_i$ is a Dedekind ring. Our conditions imply that $E_n/E$ is {\sl unramified} at all $\nu_i$ for all $i<n$. This means that if $\mathcal{P}_i$ is a maximal ideal of $\tilde{O}_i$ that contains $\pi_i \tilde{O}_i$ (such an ideal always exists) and 
$$\mathrm{ord}_{\mathcal{P}_i}:E_n^{*}\twoheadrightarrow \Z$$
is the discrete valuation map attached to $\mathcal{P}_i$ then the restriction of
$\mathrm{ord}_{\mathcal{P}_i}$ to $E^{*}$ coincides with $\nu_i$. This implies that for all positive integers $i,k \le n-1$
$$\mathrm{ord}_{\mathcal{P}_i}(\pi_k)=\nu_i(\pi_k)=\delta_{ik}.$$
In particular,
$$\mathrm{ord}_{\mathcal{P}_i}(\pi_i)=\nu_i(\pi_i)=1,$$
i.e, $\pi_i$ is a uniformizer in the corresponding discrete valuation (sub)ring 
$O_{\mathrm{ord}_{\mathcal{P}_i}}$ 
of $E_n$ attached to $\mathrm{ord}_{\mathcal{P}_i}$. Now  the induction assumption applied to $E_n$ and its $(n-1)$ discrete valuation maps $\mathrm{ord}_{\mathcal{P}_i}$ ($1\le i \le n-1$) implies that $B/E_n$ is a field extension of degree $\prod_{i=1}^{n-1} p_i$.  This implies that the degree
$$[B:E]=[B:E_n][E_n:E]=\left(\prod_{i=1}^{n-1} p_i\right)p_n=\prod_{i=1}^{n} p_i.$$
This means that the $E$-dimension of $B$ is $\prod_{i=1}^{n} p_i$ and therefore the image of $S$ is a basis of the $E$-vector space $B$.
\end{proof}

\begin{cor}
\label{irredExt}
We keep the notation and assumptions of Lemma \ref{fieldExt}.
Let $R$ be a subring of $E$ that contains $1$ and all $\pi_i$ ($1 \le i \le n$).
Then the quotient  $B_R=R[T_1, \dots , T_n]/(T_1^{p_1}-\pi_1, \dots , T_n^{p_n}-\pi_n)$ of the polynomial $R$-algebra  $R[T_1, \dots , T_n]$ by the ideal generated by all $T_i^{p_i}-\pi_i$ has no zero divisors.
\end{cor}

\begin{proof}
There are the natural homomorphisms of $R$-algebras
$$R[T_1, \dots T_n]\twoheadrightarrow  B_R\to B$$
such that the first homomorphism is surjective and the {\sl injective} image of 
$$S\subset R[T_1, \dots T_n]\subset E[T_1, \dots T_n]$$ 
in $B$ is a basis of the $E$-vector space $B$. On the other hand, the image of $S$ generates $B_R$ as $R$-module. It suffices to prove that $B_R \to B$ is injective, since $B$ is a field by Lemma \ref{fieldExt}.

 Suppose that $u \in B_R$ goes to $0$ in $B$. Clearly, $u$ is a linear combination of (the images of) elements of $S$ with coefficients in $R$. Since the image of $u$ in $B$ is $0$, all these coefficients are zeros, i.e., $u=0$ in $B_R$. 
\end{proof}

\begin{ex}
\label{irredTildeC}
We use the notation of Section \ref{quadrics}.
Let us put $n=2g+1, R=K[x],  E=K(x), \pi_i=x-\alpha_i,  p_i=2$  and let
$$\nu_i: E^{*}=K(x)^{*} \twoheadrightarrow \Z$$
be the discrete valuation map of the field of rational functions $K(x)$ attached to $\alpha_i$.
Then $K[\tilde{\CC}]=B_R/\mathcal{N}(B_R)$ where $\mathcal{N}(B_R)$ is the nilradical of $B_R$. It follows from Corollary \ref{irredExt} that $\mathcal{N}(B_R)=\{0\}$ and $K[\tilde{\CC}]$ has no zero divisors, i.e., $\tilde{\CC}$ is irreducible.

\end{ex}

\end{document}